\documentclass[secthm,seceqn,amsthm,ussrhead,12pt]{amsart}
\usepackage[utf8]{inputenc}
\usepackage[english]{babel}
\usepackage{amssymb,amsmath,amsthm,amsfonts,xcolor,enumerate,hyperref,comment,longtable,cleveref}

\usepackage{times}
\usepackage{cite}
\usepackage{pdflscape}
\usepackage{ulem}
\usepackage[mathcal]{euscript}
\usepackage{tikz}
\usepackage{hyperref}
\usepackage{cancel}
\usepackage{stmaryrd}
\usepackage{MnSymbol}

\newcommand{\ann}[1]{\operatorname{Ann}{(#1)}}

\usetikzlibrary{arrows}

\newtheorem*{theoremA}{Theorem A}
\newtheorem*{theoremB}{Theorem B}
\newtheorem*{theoremC}{Theorem C}

\newtheorem*{lemma}{Lemma}

\newtheorem*{corollary}{Corollary}
\newtheorem*{Remark}{Remark}

\newenvironment{Proof}[1][Proof.]{\begin{trivlist}
\item[\hskip \labelsep {\bfseries #1}]}{\flushright
$\Box$\end{trivlist}}

\usepackage{stmaryrd}
\usepackage{xcolor}

\begin{document}
\noindent{\Large  
 The geometric classification of $2$-step nilpotent algebras \\ and applications
}\footnote{ The work of the first author was completed as a part of the implementation of the development program of the Scientific and Educational Mathematical Center Volga Federal District, agreement no. 075-02-2020-1488/1.
The work is supported by  FAPESP 19/00563-1,
RFBR 20-01-00030, CNPq 302980/2019-9.
The authors thank Sergey Pchelintsev for sharing the manuscript \cite{shaf}
and  the referee who has noticed numerous inaccuracies throughout the text and suggested several significant improvements in the overall exposition of the paper.
} 

    \medskip

   {\bf  Mikhail Ignatyev$^{a}$, Ivan Kaygorodov$^{b,c}$ \& Yury Popov$^{d}$\\

    \medskip

}

{\tiny

$^{a}$ Samara National Research University, Samara, Russia

$^{b}$ CMCC, Universidade Federal do ABC, Santo Andr\'e, Brazil

$^{c}$ Moscow Center for Fundamental and Applied Mathematics, Moscow, Russia

$^{d}$  Saint Petersburg State University, Russia

\

  \medskip

   E-mail addresses: 
   
\smallskip

Mikhail Ignatyev  (mihail.ignatev@gmail.com)

    Ivan Kaygorodov (kaygorodov.ivan@gmail.com) 

Yury Popov (yurippv@gmail.com)

}

  \medskip

\noindent {\bf Abstract.}
{\it We give a geometric classification of complex $n$-dimensional $2$-step nilpotent (all, commutative and anticommutative) algebras.
Namely,  it has been found the number of irreducible components and their dimensions.
As a corollary, we have a   geometric classification of 
complex $5$-dimensional nilpotent associative algebras.
In particular, it has been proven that this variety has $11$ irreducible components and $7$ rigid algebras. }

\ 



\section*{Introduction}
 Geometric properties of a variety of algebras defined by a family of polynomial identities have been an object of study since 1970's. Gabriel described the irreducible components of the variety of $4$-dimensional unital associative algebras~\cite{gabriel}. Mazzola classified algebraically and geometrically the variety of $5$-dimesional unital associative algebras~\cite{maz79}. 
 Cibils considered rigid associative algebras with $2$-step nilpotent radical \cite{cibils}.
 Burde and Steinhoff constructed the graphs of degenerations for the varieties of    $3$-dimensional and $4$-dimensional Lie algebras~\cite{BC99}. 
 Grunewald and O'Halloran calculated all degenerations for the variety of $5$-dimensional nilpotent Lie algebras~\cite{GRH}. 
Chouhy  proved that  in the case of finite-dimensional associative algebras,
 the $N$-Koszul property is preserved under the degeneration relation~\cite{chouhy}.
Degenerations have also been used to study the level of complexity of an algebra~\cite{gorb93,kv17}.
Given algebras ${\bf A}$ and ${\bf B}$ in the same variety, we write ${\bf A}\to {\bf B}$ and say that ${\bf A}$ {\it degenerates} to ${\bf B}$, or that ${\bf A}$ is a {\it deformation} of ${\bf B}$, if ${\bf B}$ is in the Zariski closure of the orbit of ${\bf A}$ (under the base-change action of the general linear group). The study of degenerations of algebras is very rich and closely related to deformation theory, in the sense of Gerstenhaber \cite{ger63}. It offers an insightful geometric perspective on the subject and has been the object of a lot of research.
In particular, there are many results concerning degenerations of algebras of small dimensions in a  variety defined by a set of identities (see, for example, \cite{klp20,kkp20, ikm20,GRH, ale3,BC99,cal} and references therein).
One of the main problems of the {\it geometric classification} of a variety of algebras is a description of its irreducible components.
From the geometric point of view, in many cases, the irreducible components of the variety are determined by the rigid algebras, i.e., algebras whose orbit closure is an irreducible component. It is worth mentioning that this is not always the case and  Flanigan had shown that the variety of $3$-dimensional nilpotent associative algebras has an irreducible component which does not contain any rigid algebras --- it is instead defined by the closure of a union of a one-parameter family of algebras \cite{fF68}. 

The variety of $2$-step nilpotent algebras is the intersection of all varieties of algebras defined by a family of polynomial identities of degree equal or more than three
(for example,  varieties of associative, assosymmetric, Novikov, Leibniz, Zinbiel, and so on).
And it plays an important role in the geometric classification of such algebras,
because each non-$2$-step nilpotent algebra from this variety degenerates to some $2$-step nilpotent algebras.
The same situation appears in commutative and anticommutative cases.
On the other hand, the variety of $2$-step nilpotent Lie algebras has a proper interest (for example, see \cite{lo14, aut, fb} and references therein).
A systematic study of $2$-step nilpotent algebras from the geometric point started from a paper of Shafarevich, where he described number and dimensions of irreducible components of the variety of $2$-step nilpotent commutative algebras \cite{shaf}.
Recently, it was proven that the variety of $n$-dimensional (all, commutative or anticommutative) nilpotent algebras is irreducible and the dimensions of these varieties were also calculated \cite{kkl21}.
The full graphs of degenerations 
in the varieties of  all  $4$-dimensional $2$-step nilpotent algebras  \cite{kppvcor} and
 all  $8$-dimensional $2$-step nilpotent anticommutative algebras  \cite{ale3} are constructed. 
In the first part of our paper, we are following the ideas from  \cite{shaf} for finding the number and dimensions of irreducible components in the variety of $n$-dimensional $2$-step (all, commutative and anticommutative) nilpotent algebras (Theorem A). 
After that, we are discussing in detail the case of $5$-dimensional $2$-step nilpotent algebras.
Algebraic classification of $5$-dimensional $2$-step nilpotent algebras is a very "wild" problem,
but using the results from the first part of our paper, it will be possible to take a geometric classification of these algebras (Theorem B), 
which is the principal tool for future geometric classifications of $5$-dimensional (nilpotent, solvable or all) 
associative, antiassociative, assosymmetric, bicommutative, Novikov, symmetric Leibniz, Zinbiel, etc. varieties of algebras.

Recently, an algebraic classification of complex $5$-dimensional nilpotent (non-$2$-step nilpotent) associative algebras was obtained in \cite{karim20}.
Early, a tentative geometric classification of these algebras was given in \cite{mak93},
but unfortunately, as we can see from our result, it was completely wrong.
In the present paper, using a geometric classification of complex $5$-dimensional $2$-step nilpotent algebras (Theorem B) and an algebraic classification of  
complex $5$-dimensional  nilpotent (non-$2$-step nilpotent) associative algebras \cite{karim20},
we give the complete geometric classification of complex  $5$-dimensional  nilpotent associative algebras (Theorem C).

\section{Around some results of Shafarevich} 

For $k \le n$ consider the (algebraic) subset $\mathfrak{Nil}_{n,k}^2$ of the variety $\mathfrak{Nil}_n^2$ of 2-step nilpotent $n$-dimensional algebras defined by 
\[
\mathfrak{Nil}_{n,k}^2 = \{A \in \mathfrak{Nil}_n^2 : \dim A^2 \le k, \ \dim \ann A \ge k \}.
\]
It is easy to see that $\mathfrak{Nil}_n^2 = \cup_{k=1}^n\mathfrak{Nil}_{n,k}^2.$ Analogously for the varieties $\mathfrak{Nil}_n^2{}^{\mathrm{c}}, \mathfrak{Nil}_n^2{}^{\mathrm{ac}}$ of commutative and and anticommutative 2-step nilpotent algebras we define the subsets $\mathfrak{Nil}_{n,k}^2{}^{\mathrm{c}}$ and $\mathfrak{Nil}_{n,k}^2{}^{\mathrm{ac}},$ respectively.

This theorem is completely analogous to \cite[Theorem 1]{shaf}, but we provide its proof to ensure that the paper is self-contained.

\begin{theoremA}
\label{th:theogalv}
The sets $\mathfrak{Nil}_{n,k}^2$ are irreducible and 
\begin{equation}
\label{eq:nil2_decomp}
\mathfrak{Nil}_n^2 = \bigcup_{k}\mathfrak{Nil}_{n,k}^2, \quad \text{for} \quad 1 \le k \le \left\lfloor n + \frac{1 - \sqrt{4n+1}}{2}\right\rfloor
\end{equation}
is the decomposition of $\mathfrak{Nil}_n^2$ into irreducible components. Analogously, we have decompositions: 
\begin{longtable}{lclll}
$\mathfrak{Nil}_n^2{}^{\mathrm{c}}$ &$=$& $\bigcup_{k}\mathfrak{Nil}_{n,k}^2{}^{\mathrm{c}},$ & for & $1 \le k \le \left\lfloor n + \frac{3 - \sqrt{8n+9}}{2}\right\rfloor,$\\
$\mathfrak{Nil}_n^2{}^{\mathrm{ac}}$ &$=$ & $\bigcup_{k}\mathfrak{Nil}_{n,k}^2{}^{\mathrm{ac}},$ & for  & $1 + (n + 1) \operatorname{mod} 2 \le k \le \left\lfloor n + \frac{1 - \sqrt{8n+1}}{2}\right\rfloor \text{ for } n \ge 3.$\\
\end{longtable}
Moreover,
\begin{longtable}{lcl}
$\dim \mathfrak{Nil}_{n,k}^2$ &$=$ &$(n-k)^2k + (n-k)k,$\\
$\dim \mathfrak{Nil}_{n,k}^2{}^{\mathrm{c}}$ &$=$ & $\frac{(n-k)(n-k+1)}{2}k + (n-k)k,$\\
$\dim \mathfrak{Nil}_{n,k}^2{}^{\mathrm{ac}}$ &$=$ & $\frac{(n-k)(n-k-1)}{2}k + (n-k)k.$
\end{longtable}
\end{theoremA}
\begin{proof}
We prove the result for the variety $\mathfrak{Nil}_n^2,$ for the varieties $\mathfrak{Nil}_n^2{}^{\mathrm{c}}$ and $\mathfrak{Nil}_n^2{}^{\mathrm{ac}}$ the proofs are analogous. Let $A \in \mathfrak{Nil}_n^2$ be an algebra and let $k = \dim A^2.$ Since the square of $A$ lies in its annihilator, it must be spanned by other $n-k$ basis vectors. Therefore, we must have $k \le (n-k)^2$ and we have the decomposition \eqref{eq:nil2_decomp}.

Let us show that the sets $\mathfrak{Nil}_{n,k}^2$ are irreducible. For $k$ as in \eqref{eq:nil2_decomp} consider the set $S_{n,k} \subset \mathfrak{Nil}_{n,k}^2$ of algebras with the following multiplication tables:
\begin{equation}
\begin{gathered}
\label{eq:Snk_mult_table}
e_ie_j = \sum_{p=1}^kc_{ij}^{n-k+p}e_{n-k+p},\quad
e_ie_{n-k+p} = 0,\quad
e_{n-k+p}e_{n-k+s} = 0,\\
i,j = 1,\ldots, n-k,\quad
p,s = 1,\ldots,k.
\end{gathered}
\end{equation}
Because $S_{n,k}$ is irreducible (indeed, it is isomorphic to $\mathbb{C}^{(n-k)^2k}$), it must lie in a unique irreducible component of $\mathfrak{Nil}_{n,k}^2.$ However, it is easy to see that every algebra from $\mathfrak{Nil}_{n,k}^2$ is $\operatorname{GL}_n(\mathbb{C})$-conjugated to an algebra from $\mathfrak{Nil}_{n,k}^2$ (that is, one can choose a basis in that algebra such that the multiplication table in this basis is of the form \eqref{eq:Snk_mult_table}). Since the group $\operatorname{GL}_n(\mathbb{C})$ is connected, its action on $\mathfrak{Nil}_{n,k}^2$ must preserve irreducible components. This shows that $\mathfrak{Nil}_{n,k}^2$ is itself irreducible.

Consider a subset $U_{n,k} \subset \mathfrak{Nil}_{n,k}^2$ given by $U_{n,k} = \{A \in \mathfrak{Nil}_n^2 : \dim A^2 = k,\ \dim \ann A = k \}.$ As we have seen, the sets $\{A \in \mathfrak{Nil}_n^2 : \dim A^2 \le l\}$ and $\{A \in \mathfrak{Nil}_n^2 : \dim \ann A \ge s\}$ are algebraic for all $l$ and $s,$ thus the set $U_{n,k}$ is open in $\mathfrak{Nil}_{n,k}^2.$ However, the sets $U_{n,k}$ and $U_{n,k'}$ have empty intersection for $k' \ne k,$ therefore $\mathfrak{Nil}_{n,k}^2 \not\subseteq \cup_{k'\ne k}\mathfrak{Nil}_{n,k'}^2$ and the decomposition \eqref{eq:nil2_decomp} is indeed the decomposition of $\mathfrak{Nil}_n^2$ into irreducible components. The same proof is valid for $\mathfrak{Nil}_n^2{}^c.$ However, the set $U_{n,1} \cap \mathfrak{Nil}^2{}^{\mathrm{ac}}$ is empty for an even $n$: indeed, let $A^2 = \langle e \rangle$ for $e \in A,$ where $(A,\mu)$ is an anticommutative 2-step nilpotent algebra. Then the matrix corresponding to the map $\operatorname{proj}_e \circ \mu: V \otimes V \to \mathbb{F},$ where $\operatorname{proj}_e$ is the projection to the space generated by $e$ and $V$ is a vector space complement of $\langle e \rangle$ to $A$ is skew-symmetric and is of odd size, therefore it must be degenerate. This means that there is a nonzero vector in $V \cap \ann A$ and $\dim \ann A \ge 2.$ Particularly, we must have $\mathfrak{Nil}_{n,1}^2{}^{\mathrm{ac}} \subseteq \mathfrak{Nil}_{n,2}^2{}^{\mathrm{ac}}$ if $n \ge 3.$ However, for any $n, 2 \le k \le (n-k)(n-k-1)/2$ the sets $U_{n,k} \cap \mathfrak{Nil}^2{}^{\mathrm{ac}}$ are nonempty (and open in $\mathfrak{Nil}_{n,k}^2{}^{\mathrm{ac}}$). Indeed, if $A^2 = \langle e_{n-k+1}, \ldots, e_n \rangle$ where $k$ is as above and $V$ is a vector space complement of $A^2$ to $A$, then considering the matrices $A_k$ corresponding to the maps $\operatorname{proj}_{e_{n-k+p}} \circ \mu: V \otimes V \to \mathbb{F},\ p = 1, \ldots, k,$ one can see that the condition $A \in U_{n,k}$ is equivalent to the condition that matrices $A_k$ are linearly independent and their kernels have zero intersection. Choosing an appropriate basis in the space of the skew-symmetric matrices of size $(n-k),$ it is easy to see that this condition indeed defines an open subset. 

Let us now calculate the dimension of $\mathfrak{Nil}_{n,k}^2.$ Considering the map
\begin{align*}
\operatorname{GL}_n(\mathbb{C}) \times S_{n,k} &\to \operatorname{GL}_n(\mathbb{C}) \cdot S_{n,k} = \mathfrak{Nil}_{n,k}^2,\\
(g,A) &\mapsto g\cdot A,
\end{align*}
we get 
\begin{equation}
\label{eq:dim_by_map}
\dim \mathfrak{Nil}_{n,k}^2 = \dim S_{n,k} + \dim \operatorname{GL}_n(\mathbb{C}) - \dim F, 
\end{equation}
where $F$ is the preimage of a generic point of $\mathfrak{Nil}_{n,k}^2$ with respect to this mapping. Consider the preimage of an algebra $A \in S_{n,k} \cap U_{n,k}.$ If $g\cdot A' = A$ for some $g \in \operatorname{GL}_n(\mathbb{C}),\ A' \in S_{n,k},$ then in the basis $g^{-1}(e_1), \ldots, g^{-1}(e_{n-k}), g^{-1}(e_{n-k+1}), \ldots, g^{-1}(e_n)$ the algebra $A'$ has mutliplication table of the form \eqref{eq:Snk_mult_table}. But since $A \in U_{n,k},$ we must have $A^2 = \langle e_{n-k+1}, \ldots, e_n \rangle = \langle g^{-1}(e_{n-k+1}), \ldots, g^{-1}(e_n) \rangle.$ Therefore, $g$ must lie in the group preserving the space $A^2$ whose dimension is $n^2 - k(n-k)$ (note that it acts transitively on $A^2$). Substituting all dimensions in \eqref{eq:dim_by_map}, we get the dimension of $\mathfrak{Nil}_{n,k}^2$ as stated.
\end{proof}

\section{The geometric classification of complex $5$-dimensional $2$-step nilpotent algebras}

\subsection{$\mathfrak{Nil}_{n+1,1}^2$}
Let $A$ be an algebra from $\mathfrak{Nil}_{n+1,1}^2.$
If $A$ has the multiplication table $e_ie_j= {\mathfrak a}_{ij} e_{n+1},$ 
then it can be identified with an $n \times n$ matrix ${\bf A}= ({\mathfrak a}_{ij}).$
The problem of description of isomorphism classes of $\mathfrak{Nil}_{n+1,1}^2$ is then reduced to the calculation of equivalence classes of $n\times n$ matrices under congruence, which was considered in \cite{hs08}.
Thanks to \cite[Theorem 2.1]{hs08}, 
over the complex field, every
square matrix is congruent to a direct sum
of matrices of the form (determined uniquely up to permutation of summands, see \cite{hs08}):
\begin{enumerate}
    \item $J_k(0);$  

    \item $ [ J_k(\lambda)\backslash I_k ];$

    \item $\Gamma_k,$ 
\end{enumerate}
where $J_k(\lambda)$ is the Jordan block of size $k$ corresponding to the eigenvalue $\lambda,$ $I_k$ is the identity matrix of size $k,$ $[A \backslash B] = \begin{pmatrix}0 & B \\ A & 0\end{pmatrix}$ is the skew sum of matrices $A$ and $B$ and 
\[
\Gamma_n = 
\begin{pmatrix}
0 & & & & & \udots\\
& & & & 1 & \udots \\
& & & -1 & 1 & \\
& & 1 & 1 & & \\
& -1 & -1 & & & \\
1 & 1 & & & &
\end{pmatrix}.
\]
Recalling that $[J_1(\lambda_i)\backslash I] =
\left(\begin{array}{cc}
0 & 1\\ 
\lambda_i &0 \end{array} \right)$ and  
let us define a matrix ${\mathfrak A}_n$ as 
\begin{enumerate}
    \item $
{\mathfrak A}_{2t}={\rm diag} \{[J_1(\lambda_1)\backslash I], \ldots, [J_1(\lambda_t)\backslash I]\},$
\item $
{\mathfrak A}_{2t+1}= {\rm diag} \{[J_1(\lambda_1)\backslash I], \ldots, [J_1(\lambda_t)\backslash I], \Gamma_1\}.$
\end{enumerate}
Then ${\bf H}(\lambda_1, \ldots, \lambda_{\lfloor{n/2} \rfloor})$ is the  $(n+1)$-dimensional algebra defined by the matrix ${\mathfrak A}_n.$ 
After a careful calculation of the dimension of the space of derivations of all algebras from this family we conclude that $\dim \operatorname{Der}{\bf H}(\lambda_1, \ldots, \lambda_{\left\lfloor \frac{n}{2} \right\rfloor})$ is generically $\left\lfloor \frac{3n}{2} \right\rfloor + 1,$ therefore, the dimension of algebraic variety which defined this family is 
\[
(n+1)^2 - \dim \operatorname{Der}{\bf H}(\lambda_1, \ldots, \lambda_{\left\lfloor \frac{n}{2} \right\rfloor}) + \left\lfloor \frac{n}{2} \right\rfloor = n(n+1)
\]
Hence, \cref{th:theogalv} gives the following obvious corollary

\begin{corollary}
\label{cor:niln1}
The variety $\mathfrak{Nil}_{n,1}^2$ is  defined by the family of algebras 
${\bf H}(\lambda_1, \ldots, \lambda_{\left\lfloor \frac{n}{2} \right\rfloor}).$
In particular,  $\mathfrak{Nil}_{5,1}^2$  is defined by the following $2$-parameters family of algebras
\[{\bf H}(\lambda, \mu) \ : \ e_1e_2=e_5 \ \ e_2e_1=\lambda e_5 \ \ e_3e_4=e_5 \ \ e_4e_3=\mu e_5,
\]
which  will be called ${\mathfrak V}_{4+1}.$
\end{corollary}

\subsection{$\mathfrak{Nil}_{5,2}^2$}
Let $A$ be an algebra from $\mathfrak{Nil}_{5,2}^2,$
then if $A$ has the following multiplication table 
\begin{center}
$e_ie_j= {\mathfrak a}_{ij} e_{4}+{\mathfrak b}_{ij} e_{5}$ \ $(1 \leq i,j \leq 3)$,
\end{center}
  it can be identified with two $3 \times 3$ matrices 
${\bf A}= ({\mathfrak a}_{ij})$ and 
${\bf B}= ({\mathfrak b}_{ij}).$
Let us denote $A$ as $\langle\langle{\bf A}:{\bf B}\rangle\rangle.$
The problem of description of isomorphism classes of $\mathfrak{Nil}_{5,2}^2$ is 
reduced to the description of equivalence classes of 
vector spaces generated by two 
$3\times 3$ matrices under the action of ${\rm GL}_3.$ 
    
The variety $\mathfrak{Nil}_{5,2}^2$ is defined by an $18$-dimensional family of algebras constructed from matrices ${\bf A}= ({\mathfrak a}_{ij})$ and 
${\bf B}= ({\mathfrak b}_{ij}),$ where ${\mathfrak a}_{ij}$ and ${\mathfrak b}_{ij}$ are parameters.
Thanks to \cref{cor:niln1} and \cite{hs08}, the family of algebras  
$\langle\langle{\bf A}:{\bf B}\rangle\rangle$ is in the orbit closure of the
$10$-parameters family of algebras 
$\langle\langle \left(\begin{array}{ccc}
1 & 0 & 0\\ 
0& 0 &1\\ 
0& \lambda &0\\ 
\end{array} \right) :{\bf B}\rangle\rangle,$
where $\lambda$ and ${\mathfrak b}_{ij}$ are parameters. 
Obviously, the last variety is in the orbit closure of the $8$-dimensional family of algebras
$\langle\langle 
\left(\begin{array}{ccc}
1 & 0 & 0\\ 
0& 0 &1\\ 
0& \lambda &0\\ 
\end{array} \right) :
\left(\begin{array}{ccc}
0 & \mu_1 & \mu_2\\ 
\mu_3& \mu_4 &\mu_5\\ 
\mu_6& \mu_7 & 1\\ 
\end{array} \right) 
\rangle\rangle,$ where $\lambda$ and $\mu_i$ are parameters.
The last family of algebras, which has the following multiplication table

\begin{longtable}{lllllll}
  & &
$e_1e_1 =  e_4$& $e_1e_2 = \mu_1 e_5$ & $e_1e_3 =\mu_2 e_5$& 
$e_2e_1 = \mu_3 e_5$  & $e_2e_2 = \mu_4 e_5$  \\
& & $e_2e_3 = \mu_5 e_5$  & $e_3e_1 = \mu_6 e_5$  & 
$e_3e_2 = \lambda e_4+ \mu_7 e_5$ & $e_3e_3 =  e_5,$  \\
\end{longtable}
which  will be  denote as $\mathfrak V_{3+2}.$
Hence, the variety $\mathfrak{Nil}_{5,2}^2$ is defined by   ${\mathfrak V}_{3+2}.$

\begin{Remark}
After a careful calculation of the dimension of the algebra of derivations of all algebras from   ${\mathfrak V}_{3+2},$ we conclude, that between its $8$ parameters, there are $6$ independent parameters and $2$ dependent of  others.
\end{Remark}

\subsection{$\mathfrak{Nil}_{5,3}^2$}

Thanks to \cite{cal} we have a classification of all non-isomorphic non-split algebras from the variety  $\mathfrak{Nil}_{5,3}^2.$
As we  know, it is an irreducible variety.
After a carefully checking of their geometric dimensions, 
we have that $\mathfrak{Nil}_{5,3}^2$ is defined by the following family of algebras
\[
A_{133}(\lambda)  \ : \  e_1e_1 = e_3 + \lambda e_5 \ \  e_1e_2 = e_3 \ \ e_2e_1 = e_4 \ \ e_2e_2 = e_5,
\]
which will be called ${\mathfrak V}_{2+3}.$

\subsection{Classification theorem} Summarizing all ideas from the present section, we have the following statement.
\begin{theoremB} 
The variety of complex $5$-dimensional $2$-step nilpotent   algebras  has 
dimension $24$  and it has 
$3$ irreducible components
defined by  
  
\begin{center}
$\mathcal{C}_1=\overline{\{\mathcal{O}({\mathfrak V}_{2+3} )\}},$ \ 
$\mathcal{C}_2=\overline{\{\mathcal{O}({\mathfrak V}_{3+2})\}},$ \
$\mathcal{C}_3=\overline{\{\mathcal{O}({\mathfrak V}_{4+1})\}}.$

\end{center}

\end{theoremB}

\section{The geometric classification of complex $5$-dimensional  nilpotent associative algebras}

\subsection{Algebraic classication of complex $5$-dimensional nilpotent associative algebras}
Thanks to \cite{karim20,kaladra20, ikm20}, we have the classification of all 
complex $5$-dimensional nilpotent (non-$2$-step nilpotent) non-commutative associative algebras:
 
\begin{longtable}{llllllllllllllllll}

${\mathcal A}^4_{05}$  & $:$  &
$e_1 e_1 = e_2$ &  $ e_1e_2=e_4$ & $e_1e_3= e_4$ & $e_2e_1= e_4$ & $e_3e_3=e_4$\\ \hline

${\mathcal A}^4_{06}(1) $ & $:$   &
$e_1 e_1 = e_2$ & $e_1e_2=e_4$ & $e_1e_3=e_4$ & $e_2e_1= e_4$ \\  



\hline
$\mu_{1,3}^5$ &$: $&
$ e_1e_1=e_2$ &$ e_1e_2=e_3$ &$ e_1e_3=e_4$ & $e_1e_5=e_{4}$ \\&    
&$ e_2e_1=e_3$ &$ e_2e_2=e_4$&$ e_3e_1=e_4$
 
\\

\hline
$\mu_{1,4}^5$ &$: $&
$ e_1e_1=e_2$ &$ e_1e_2=e_3$ &$ e_1e_3=e_4$ & $e_1e_5=e_{4}$ \\&    
&$ e_2e_1=e_3$ &$ e_2e_2=e_4$&$ e_3e_1=e_4$  & $e_5e_5=e_{4}$  \\


\hline$\lambda_2$&$: $& $e_1e_1=e_2$& $e_1e_2=e_3$& $e_1e_4=e_3$\\&& $e_2e_1=e_3$ & $e_4e_5=e_3$ &$e_5e_4=e_3$ \\

\hline
$\lambda_3$&$: $& $e_1e_1=e_2$ & $e_1e_2=e_3$  & $e_1e_4=e_3$ &$e_2e_1=e_3$& $e_5e_5=e_3$\\[1mm]

\hline
$\lambda_4$&$:$ & $e_1e_1=e_2$& $e_1e_2=e_3$& $e_1e_4=e_3$ \\&&$e_2e_1=e_3$   & $e_4e_4=e_3$ & $e_5e_5=e_3$ \\

\hline
$\lambda_5$&$: $& $e_1e_1=e_2$& $e_1e_2=e_3$&$e_2e_1=e_3$ \\&&
$e_4e_5=e_3$ & $e_5e_4=-e_3$ & $e_5e_5=e_3$\\ 

\hline
$\lambda_6^{\alpha\neq 1}$&$:$ & $e_1e_1=e_2$& $e_1e_2=e_3$ & $e_2e_1=e_3$ & $e_4e_5=e_3$ & $e_5e_4=\alpha e_3$ \\

\hline $\mu_ 1$&$:$&$e_1e_1=e_2$&$e_1e_2=e_3$&$ e_2e_1=e_3$&$e_4e_1=e_5$\\

\hline $\mu_ 2$&$:$&$e_1e_1=e_2$&$e_1e_2=e_3$&$e_2e_1=e_3$&$e_4e_1=e_5$&$e_4e_4=e_3$\\

\hline $\mu_ 3$&$:$&$e_1e_1=e_2$&$e_1e_2=e_3$&$e_2e_1=e_3$\\&&
$e_4e_1=e_5$&$e_4e_2=e_3$&$e_5e_1=e_3$\\

\hline $\mu_ 4$&$:$&$e_1e_1=e_2$&$e_1e_2=e_3$&$e_2e_1=e_3$&$e_4e_1=e_5$\\
&&$e_4e_2=e_3$&$e_4e_4=e_3$ &$e_5e_1=e_3$  \\

\hline $\mu_ 5$&$:$&$e_1e_1=e_2$&$e_1e_2=e_3$&$e_1e_4=e_5$&$e_2e_1=e_3$&$e_4e_1=e_3+e_5$\\

\hline $\mu_ 6$&$:$&$e_1e_1=e_2$&$e_1e_2=e_3$&$e_1e_4=e_5$\\
&&$e_2e_1=e_3$&$e_4e_1=e_3+e_5$&$e_4e_4=e_3$\\

\hline $\mu_ 7^{\alpha\neq 1}$
&$:$&$e_1e_1=e_2$&$e_1e_2=e_3$ &$e_1e_4=e_5$&$e_2e_1=e_3$&$e_4e_1=\alpha e_5$\\

\hline $\mu_ 8^{\alpha\neq 1}$&$:$&$e_1e_1=e_2$&$e_1e_2=e_3$&$e_1e_4=e_5$\\
&&$e_2e_1=e_3$&$e_4e_1=\alpha e_5$&$e_4e_4=e_3$\\

\hline $\mu_ 9$&$:$&$e_1e_1=e_2$&$e_1e_2=e_3$&$e_2e_1=e_3$&$e_4e_1=e_3$&$e_4e_4=e_5$\\

\hline $\mu_ {10}$&$:$&$e_1e_1=e_2$&$e_1e_2=e_3$&$e_1e_4=e_5$&$e_2e_1=e_3$&$e_4e_4=e_5$\\

\hline $\mu_ {11}$&$:$&$e_1e_1=e_2$&$e_1e_2=e_3$&$e_1e_4=e_5$&$e_2e_1=e_3$&$e_4e_4=e_3+e_5$\\

\hline $\mu_ {12}$&$:$&$e_1e_1=e_2$&$e_1e_2=e_3$&$e_1e_4=e_5$\\&&$e_2e_1=e_3$&$e_4e_1=e_2-e_5$&$e_5e_1=e_3$\\

\hline $\mu_ {13}$&$:$&$e_1e_1=e_2$&$e_1e_2=e_3$&$e_1e_4=e_5$&$e_2e_1=e_3$\\
&&$e_4e_1=e_2-e_5$&$e_4e_4=e_3$&$e_5e_1=e_3$\\

\hline $\mu_ {14}$&$:$&$e_1e_1=e_2$&$e_1e_2=e_3$&$e_1e_4=e_5$&$e_2e_1=e_3$\\
&&$e_4e_1=e_2+e_5$&$e_4e_2=2e_3$&$e_4e_4=2e_5$&$e_5e_1=e_3$\\

\hline $\mu_ {15}$&$:$&$e_1e_1=e_2$&$e_1e_2=e_3$&$e_1e_4=e_5$&$e_2e_1=e_3$\\&&$e_4e_1=e_2+e_5$&$e_4e_2=2e_3$&$
e_4e_4=e_3+2e_5$&$e_5e_1=e_3$\\


\hline $\mu_ {17}$&$:$&$e_1e_1=e_2$&$e_1e_2=e_3$&$e_1e_4=e_5$&$e_2e_1=e_3$\\
&&$e_4e_1=e_3+e_5$&$e_4e_4=e_2$&$e_4e_5=e_3$&$e_5e_4=e_3$\\

\hline $\mu_ {18}$&$:$&$e_1e_1=e_2$&$e_1e_2=e_3$&$ e_1e_4=e_5$&$e_2e_1=e_3$\\
&&$e_4e_1=-e_5$&$e_4e_4=e_2$&$e_4e_5=-e_3$&$e_5e_4=e_3$\\

\hline $\mu_ {19}$&$:$&$e_1e_1=e_2$&$e_1e_2=e_3$ &$e_1e_4=e_5$&$e_2e_1=e_3$\\&&$e_4e_1=e_2$&$e_4e_2=e_3$&$e_4e_4=e_3+e_5$&$e_5e_1=e_3$\\

\hline $\mu_ {20}$&$:$&$e_1e_1=e_2$&$e_1e_2=e_3$&$e_1e_4=e_5$&$e_2e_1=e_3$\\&&$ e_4e_1=e_3+e_5$&$e_4e_4=-e_2+2e_5$&$e_4e_5=e_3$&$e_5e_4=-e_3$\\

\hline $\mu_ {21}^i$&$:$&
$e_1e_1=e_2$&$e_1e_2=e_3$&$e_1e_4=e_5$&$e_2e_1=e_3$\\&&
\multicolumn{2}{l}{$e_4e_1=(1-i)e_2+i e_5$}&$
e_4e_2=2e_3$&
\multicolumn{2}{l}{$e_4e_4=-i e_2+e_3+(1+i)e_5$}\\
&&$
e_4e_5=e_3$&$
e_5e_1=(1-i)e_3$&$
e_5e_4=-i e_3$\\

\hline $\mu_ {21}^{-i}$&$:$&
$e_1e_1=e_2$&$e_1e_2=e_3$&$e_1e_4=e_5$&$e_2e_1=e_3$\\&&
\multicolumn{2}{l}{$e_4e_1=(1+i)e_2-i e_5$}&$
e_4e_2=2e_3$&
\multicolumn{2}{l}{$e_4e_4=i e_2+e_3+(1-i)e_5$}\\
&&$
e_4e_5=e_3$&$
e_5e_1=(1+i)e_3$&$
e_5e_4=i e_3$ \\

\hline $\mu_ {22}^\alpha$&$:$&$ e_1e_1=e_2$&$  e_1e_2=e_3$&$e_1e_4=e_5$&$e_2e_1=e_3$\\&&
\multicolumn{2}{l}{$e_4e_1=(1-\alpha)e_2+\alpha e_5$}&$
e_4e_2=(1-\alpha^2)e_3$&
\multicolumn{2}{l}{$e_4e_4=-\alpha e_2+(1+\alpha)e_5$}\\&&
$e_4e_5=-\alpha^2e_3$&
$e_5e_1=(1-\alpha)e_3$&
$e_5e_4=-\alpha e_3$\\

\end{longtable}
  
\subsection{Geometric classication of complex $5$-dimensional nilpotent associative algebras}

\subsubsection{Degenerations of algebras}
Given an $n$-dimensional vector space ${\bf V}$, the set ${\rm Hom}({\bf V} \otimes {\bf V},{\bf V}) \cong {\bf V}^* \otimes {\bf V}^* \otimes {\bf V}$ 
is a vector space of dimension $n^3$. This space inherits the structure of the affine variety $\mathbb{C}^{n^3}.$ 
Indeed, let us fix a basis $e_1,\dots,e_n$ of ${\bf V}$. Then any $\mu\in {\rm Hom}({\bf V} \otimes {\bf V},{\bf V})$ is determined by $n^3$ structure constants $c_{i,j}^k\in\mathbb{C}$ such that
$\mu(e_i\otimes e_j)=\sum_{k=1}^nc_{i,j}^ke_k$. A subset of ${\rm Hom}({\bf V} \otimes {\bf V},{\bf V})$ is {\it Zariski-closed} if it can be defined by a set of polynomial equations in the variables $c_{i,j}^k$ ($1\le i,j,k\le n$).

The general linear group ${\rm GL}({\bf V})$ acts by conjugation on the variety ${\rm Hom}({\bf V} \otimes {\bf V},{\bf V})$ of all algebra structures on ${\bf V}$:
$$ (g * \mu )(x\otimes y) = g\mu(g^{-1}x\otimes g^{-1}y),$$ 
for $x,y\in {\bf V}$, $\mu\in {\rm Hom}({\bf V} \otimes {\bf V},{\bf V})$ and $g\in {\rm GL}({\bf V})$. Clearly, the ${\rm GL}({\bf V})$-orbits correspond to the isomorphism classes of algebras structures on ${\bf V}$. Let $T$ be a set of polynomial identities which is invariant under isomorphism. Then the subset $\mathbb{L}(T)\subset {\rm Hom}({\bf V} \otimes {\bf V},{\bf V})$ of the algebra structures on ${\bf V}$ which satisfy the identities in $T$ is ${\rm GL}({\bf V})$-invariant and Zariski-closed. It follows that $\mathbb{L}(T)$ decomposes into ${\rm GL}({\bf V})$-orbits. The ${\rm GL}({\bf V})$-orbit of $\mu\in\mathbb{L}(T)$ is denoted by $O(\mu)$ and its Zariski closure by $\overline{O(\mu)}$.

Let ${\bf A}$ and ${\bf B}$ be two $n$-dimensional algebras satisfying the identities from $T$ and $\mu,\lambda \in \mathbb{L}(T)$ represent ${\bf A}$ and ${\bf B}$ respectively.
We say that ${\bf A}$ {\it degenerates} to ${\bf B}$ and write ${\bf A}\to {\bf B}$ if $\lambda\in\overline{O(\mu)}$.
Note that in this case we have $\overline{O(\lambda)}\subset\overline{O(\mu)}$. Hence, the definition of a degeneration does not depend on the choice of $\mu$ and $\lambda$. It is easy to see that any algebra degenerates to the algebra with zero multiplication. If ${\bf A}\to {\bf B}$ and ${\bf A}\not\cong   {\bf B}$, then ${\bf A}\to {\bf B}$ is called a {\it proper degeneration}. We write ${\bf A}\not\to {\bf B}$ if $\lambda\not\in\overline{O(\mu)}$ and call this a {\it non-degeneration}. Observe that the dimension of the subvariety $\overline{O(\mu)}$ equals $n^2-\dim {\mathfrak Der}({\bf A})$. Thus if ${\bf A}\to {\bf B}$ is a proper degeneration, then we must have $\dim{\mathfrak Der}({\bf A})>\dim{\mathfrak Der}({\bf B})$.

Let ${\bf A}$ be represented by $\mu\in\mathbb{L}(T)$. Then  ${\bf A}$ is  {\it rigid} in $\mathbb{L}(T)$ if $O(\mu)$ is an open subset of $\mathbb{L}(T)$.
Recall that a subset of a variety is called {\it irreducible} if it cannot be represented as a union of two non-trivial closed subsets. A maximal irreducible closed subset of a variety is called an {\it irreducible component}.
It is well known that any affine variety can be represented as a finite union of its irreducible components in a unique way.
The algebra ${\bf A}$ is rigid in $\mathbb{L}(T)$ if and only if $\overline{O(\mu)}$ is an irreducible component of $\mathbb{L}(T)$.

In the present work we use the methods applied to Lie algebras in \cite{GRH}.
To prove 
degenerations, we will construct families of matrices parametrized by $t$. Namely, let ${\bf A}$ and ${\bf B}$ be two algebras represented by the structures $\mu$ and $\lambda$ from $\mathbb{L}(T)$, respectively. Let $e_1,\dots, e_n$ be a basis of ${\bf V}$ and $c_{i,j}^k$ ($1\le i,j,k\le n$) be the structure constants of $\lambda$ in this basis. If there exist $a_i^j(t)\in\mathbb{C}$ ($1\le i,j\le n$, $t\in\mathbb{C}^*$) such that the elements $E_i^t=\sum_{j=1}^na_i^j(t)e_j$ ($1\le i\le n$) form a basis of ${\bf V}$ for any $t\in\mathbb{C}^*$, and the structure constants $c_{i,j}^k(t)$ of $\mu$ in the basis $E_1^t,\dots, E_n^t$ satisfy $\lim\limits_{t\to 0}c_{i,j}^k(t)=c_{i,j}^k$, then ${\bf A}\to {\bf B}$. In this case  $E_1^t,\dots, E_n^t$ is called a {\it parametric basis} for ${\bf A}\to {\bf B}$ and it will be denote as
${\bf A}\xrightarrow{(E_1^t,\dots, E_n^t)} {\bf B}$. 
Let us denote the subalgebra generated by $\langle e_i, \ldots, e_n \rangle$ as $A_i.$ 

To prove a non-degeneration ${\bf A }\not\to {\bf B}$ we will use the following Lemma (see \cite{GRH}).

\begin{lemma}\label{main}
Let $\mathcal{B}$ be a Borel subgroup of ${\rm GL}({\bf V})$ and $\mathcal{R}\subset \mathbb{L}(T)$ be a $\mathcal{B}$-stable closed subset.
If ${\bf A} \to {\bf B}$ and ${\bf A}$ can be represented by $\mu\in\mathcal{R}$ then there is $\lambda\in \mathcal{R}$ that represents ${\bf B}$.
\end{lemma}

In particular, it follows from Lemma   that ${\bf A}\not\to {\bf B}$, whenever $\dim({\bf A}^2)<\dim({\bf B}^2)$.

When the number of orbits under the action of ${\rm GL}({\bf V})$ on  $\mathbb{L}(T)$ is finite, the graph of primary degenerations gives the whole picture. In particular, the description of rigid algebras and irreducible components can be easily obtained. Since the variety of $5$-dimensional nilpotent associative algebras contains infinitely many non-isomorphic algebras, we have to fulfill some additional work. Let ${\bf A}(*):=\{{\bf A}(\alpha)\}_{\alpha\in I}$ be a family of algebras and ${\bf B}$ be another algebra. Suppose that, for $\alpha\in I$, ${\bf A}(\alpha)$ is represented by a structure $\mu(\alpha)\in\mathbb{L}(T)$ and ${\bf B}$ is represented by a structure $\lambda\in\mathbb{L}(T)$. Then by ${\bf A}(*)\to {\bf B}$ we mean $\lambda\in\overline{\cup\{O(\mu(\alpha))\}_{\alpha\in I}}$, and by ${\bf A}(*)\not\to {\bf B}$ we mean $\lambda\not\in\overline{\cup\{O(\mu(\alpha))\}_{\alpha\in I}}$.

Let ${\bf A}(*)$, ${\bf B}$, $\mu(\alpha)$ ($\alpha\in I$) and $\lambda$ be as above. To prove ${\bf A}(*)\to {\bf B}$ it is enough to construct a family of pairs $(f(t), g(t))$ parametrized by $t\in\mathbb{C}^*$, where $f(t)\in I$ and $g(t)=\left(a_i^j(t)\right)_{i,j}\in {\rm GL}({\bf V})$. Namely, let $e_1,\dots, e_n$ be a basis of ${\bf V}$ and $c_{i,j}^k$ ($1\le i,j,k\le n$) be the structure constants of $\lambda$ in this basis. If we construct $a_i^j:\mathbb{C}^*\to \mathbb{C}$ ($1\le i,j\le n$) and $f: \mathbb{C}^* \to I$ such that $E_i^t=\sum_{j=1}^na_i^j(t)e_j$ ($1\le i\le n$) form a basis of ${\bf V}$ for any  $t\in\mathbb{C}^*$, and the structure constants $c_{i,j}^k(t)$ of $\mu\big(f(t)\big)$ in the basis $E_1^t,\dots, E_n^t$ satisfy $\lim\limits_{t\to 0}c_{i,j}^k(t)=c_{i,j}^k$, then ${\bf A}(*)\to {\bf B}$. In this case, $E_1^t,\dots, E_n^t$ and $f(t)$ are called a {\it parametric basis} and a {\it parametric index} for ${\bf A}(*)\to {\bf B}$, respectively. In the construction of degenerations of this sort, we will write $\mu\big(f(t)\big)\to \lambda$, emphasizing that we are proving the assertion $\mu(*)\to\lambda$ using the parametric index $f(t)$.


Through a series of degenerations summarized in the table below by the corresponding parametric bases and indices, we obtain the main result of the second part of the paper.

\subsubsection{Classification theorem}
Let us define  the complex $n$-dimensional null-filiform associative algebra 
 \begin{center}$\mu_0^n=\{ e_i e_j=e_{i+j}, 2\leq i+j \leq n\}.$ \end{center}
 It is easy to see that  $\mu_0^n$
is rigid in the variety of complex  $n$-dimensional nilpotent associative algebras.
 
\begin{theoremC} 
The variety of complex $4$-dimensional nilpotent associative  algebras  has 
dimension $13$  and it has 
$4$ irreducible components
defined by  
 
\begin{center}
$\mathcal{C}_1=\overline{\{\mathcal{O}(\mathfrak{N}_2(\alpha))\}},$ \ $\mathcal{C}_2=\overline{\{\mathcal{O}(\mathfrak{N}_3(\alpha))\}},$ \
$\mathcal{C}_3=\overline{\{\mathcal{O}(\mu_0^4)\}},$ \  $\mathcal{C}_4=\overline{\{\mathcal{O}({\mathcal A}^4_{05}) \} }.$ 
\end{center}
In particular, we have $2$ rigid algebras: $\mu_0^4$ and 
${\mathcal A}^4_{05}.$

The variety of complex $5$-dimensional nilpotent associative  algebras  has 
dimension $24$  and it has 
$11$ irreducible components
defined by  
 
\begin{center}
$\mathcal{C}_1=\overline{\{\mathcal{O}({\mathfrak V}_{3+2})\}},$ \
$\mathcal{C}_2=\overline{\{\mathcal{O}({\mathfrak V}_{4+1})\}},$ \  
$\mathcal{C}_3=\overline{\{\mathcal{O}( \mu_0^5) \} },$ 
$\mathcal{C}_4=\overline{\{\mathcal{O}( \mu_{1,4}^5) \} },$ 
$\mathcal{C}_5=\overline{\{\mathcal{O}( \lambda_6^{\alpha}) \} },$ 
$\mathcal{C}_6=\overline{\{\mathcal{O}( \mu_{11}) \} },$ 
$\mathcal{C}_7=\overline{\{\mathcal{O}( \mu_{15}) \} },$ 
$\mathcal{C}_8=\overline{\{\mathcal{O}( \mu_{17}) \} },$ 
$\mathcal{C}_{9}=\overline{\{\mathcal{O}( \mu_{18}) \} },$ 
$\mathcal{C}_{10}=\overline{\{\mathcal{O}( \mu_{20}) \} },$
$\mathcal{C}_{11}=\overline{\{\mathcal{O}( \mu_{22}^\alpha) \} }.$

\end{center}

 In particular, we have $7$ rigid algebras:
 $\mu^5_0, \mu_{1,4}^5, \mu_{11}, \mu_{15}, \mu_{17}, \mu_{18}$ and $\mu_{20}.$ 

\end{theoremC}
\begin{Proof}
 Thanks to \cite{klp20}, 
 all complex $5$-dimensional (or $4$-dimensional) nilpotent commutative associative algebras are degenerated from $\mu_0^5$ (or $\mu_0^4$).
 The variety of all complex $4$-dimensional $2$-step nilpotent algebras (see, \cite{kppvcor})
 is defined by the following families of algebras
\begin{longtable}{lllllllll}
${\mathfrak N}_2(\alpha)$  & $:$ & $e_1e_1 = e_3$ & $e_1e_2 = e_4$ & $e_2e_1 = -\alpha e_3$ & $e_2e_2 = -e_4$ \\
\hline
${\mathfrak N}_3(\alpha)$ & $:$ & $e_1e_1 = e_4$ & $e_1e_2 = \alpha e_4$ & $e_2e_1 = -\alpha e_4$ & $e_2e_2 = e_4$ & $e_3e_3=e_4$
\end{longtable}

        The algebra ${\mathcal A}^4_{05}$ (considered as a $4$-dimensional algebra) satisfies the following conditions 
\begin{longtable}{lcl}        
$\mathcal{R}_1$ & $=$& $\left\{
 A_1A_2+A_1A_2 \subseteq A_4 \right\},$\\  
$\mathcal{R}_2$ & $=$& $\left\{
 c_{23}^4=c_{32}^4\right\},$\\  
$\mathcal{R}_3$ & $=$& $\left\{ c_{22}^4c_{33}^4=c_{32}^4c_{23}^4\right\},$\\  
$\mathcal{R}_4$ & $=$& $\left\{ 
c_{23}^4 (c_{12}^4 - c_{21}^4) = c_{22}^4 (c_{13}^4 - c_{31}^4)\right\},$\\  
$\mathcal{R}_5$ & $=$& $\left\{
c_{22}^4 (c_{13}^4 - c_{31}^4)^2 c_{33}^4 +    2 c_{13}^4 (c_{23}^4)^2 c_{31}^4+
   2 c_{12}^4 c_{13}^4 c_{23}^4 c_{33}^4= 2 (c_{13}^4c_{23}^4)^2  + (c_{23}^4 c_{31}^4)^2  + (c_{12}^4 c_{33}^4)^2
   \right\},$ 
\end{longtable} 

but $\mathfrak{N}_2(\alpha)$ and $\mathfrak{N}_3(\alpha)$ do not. 
Let us prove it.
\begin{enumerate}
    \item[$\bullet$] $\mathfrak{N}_2(\alpha).$
Assume that there is some basis  $\{f_i \}_{1  \leq i \leq 4}$ such that the structure constants $\tilde {c}^k_{i,j}$ of $\mathfrak{N}_2(\alpha)$ in it satisfy all required conditions $\mathcal{R}_1 - \mathcal{R}_5.$ 
But \begin{center}
$\langle f_1,f_2,f_3,f_4\rangle \langle f_2,f_3,f_4\rangle+ \langle f_2,f_3,f_4\rangle\langle f_1,f_2,f_3,f_4\rangle$ 
\end{center}is $2$-dimensional and 
 $\mathfrak{N}_2(\alpha)$ does not satisfy the required   condition $\mathcal{R}_1.$

    \item[$\bullet$] $\mathfrak{N}_3(\alpha).$
Assume that there is some basis  $\{f_i \}_{1  \leq i \leq 4}$ such that the structure constants $\tilde {c}^k_{i,j}$ of $\mathfrak{N}_3(\alpha)$ in it satisfy all required conditions $\mathcal{R}_1 - \mathcal{R}_5.$ 

$\bullet$ Since $\mathcal{R}_1,$ it is possible to suppose that 
\begin{center}$f_4=e_4$ and $f_i=a_{i1}e_1+a_{i2}e_2+a_{i3}e_3, \, (1\leq i \leq 3).$
\end{center}

$\bullet$  $\mathcal{R}_2$ gives that  $a_{22} a_{31} = a_{21} a_{32}.$
If $(a_{22}, a_{31}, a_{21}, a_{32})=0,$ then $f_2$ and $f_3$ are linear dependent and $\{f_i\}$ is not a basis.
Hence, there is one (or more) non-zero element in this set.
Let us suppose that $a_{31}\neq 0$ (the rest cases are similar) and 
$a_{22}  = a_{21} a_{32} a_{31}^{-1}.$  

$\bullet$ 
$\mathcal{R}_3 $ gives 
$(a_{31}^2 + a_{32}^2) (a_{23} a_{31} - a_{21} a_{33})^2=0.$
It follows that $a_{31}=\pm i  a_{32}$ 
(the opposite case gives that $\{ f_i \}$ is not a basis).
Let us suppose  $a_{31}= i  a_{32}$ (the case $a_{31}=- i  a_{32}$ is similar).

$\bullet$ 
$\mathcal{R}_4$ gives
$ a_{23} (a_{11} - i a_{12})  (a_{23} a_{32} +i a_{21} a_{33}) =0$ 
and $a_{23}=0$ (all opposite cases give that  $\{ f_i \}$ is not a basis).

$\bullet$ 
$\mathcal{R}_5$ gives
    $(a_{11} - i a_{12})^2 a_{21}^2 a_{33}^4=0.$
    Here, each of all three possibilities $a_{11} = i a_{12}, $ $ a_{21}=0$ or $a_{33}=0$ gives that  $\{ f_i \}$ is not a basis.
    
    Hence,  $\mathfrak{N}_3(\alpha)$ does not satisfy the required   conditions $\mathcal{R}_1 - \mathcal{R}_5.$

\end{enumerate}

Hence
        ${\mathcal A}^4_{05} \not\to \{ \mathfrak{N}_2(\alpha), \mathfrak{N}_3(\alpha) \}.$
        Since $
        {\mathcal A}^4_{05}  \xrightarrow{(te_1, t^2e_2, t^2e_3, t^3e_4)} {\mathcal A}^4_{06}(1) 
        $  we have that the variety of complex $4$-dimensional nilpotent associative algebras is defined by
        $\mu_0^4,$ ${\mathcal A}^4_{05},$ $\mathfrak{N}_2(\alpha)$ and $\mathfrak{N}_3(\alpha).$
         Let us explain the present degeneration in detail and the rest of the degenerations we will keep without deep explication.
          We consider the following parametric basis 
                    $E_1^t=te_1, E_2^t=t^2e_2, E_3^t=t^2e_3, E_4^t=t^3e_4$ of ${\mathcal A}^4_{05}.$
                    Hence,
                    \begin{center}
                    $E_1^tE_1^t=E_2^t, \
                    E_1^tE_2^t=E_4^t, \ 
                    E_1^tE_3^t=E_4^t,\ 
                    E_2^tE_1^t=E_4^t, \ 
                    E_3^tE_3^t= t E_4^t,$
                    \end{center}
          (we omitted all zero multiplications) 
          and taking the limit for $t\to 0,$ we have a multiplication table of ${\mathcal A}^4_{06}(1),$
          which gives the degeneration $
        {\mathcal A}^4_{05}  \xrightarrow{ } {\mathcal A}^4_{06}(1).$

          \ 
          
         For the rest of our proof, we need to present the list of the following degenerations:

\begin{longtable}{|lcl|lcl|}
\hline
$\lambda_6^{ {0}}$&$
     \xrightarrow{( e_1 - e_5, e_2,  e_2 +  e_4 +  e_5, e_3, t e_5)} $&$ {\mathcal A}^4_{05}$ &

$\mu_{1,4}^5$&$  \xrightarrow{(te_1, t^2e_2, t^3e_3, t^4e_4, t^3e_5)}$&$ \mu_{1,3}^5$
\\ \hline

$\lambda_6^{ {1+t}}$&$
   \xrightarrow{(  t e_1 + e_5,  t^2 e_2, t^3 e_3,   -t  e_2 + t^2  e_4 + \frac{t^2}{2 + t} e_5,t e_5)}$&$ \lambda_2$ &

$\lambda_4 $&$ \xrightarrow{(te_1, t^2e_2, t^3 e_3, t^2e_4, t^{\frac{3}{2}}e_5)}$&$ \lambda_3$
 \\ \hline

$\mu_{1,4}^5$&$  \xrightarrow{( 
t  e_1 +  e_2 + \frac{1}{2 t} e_3, -t^2 e_2 + 2 e_4, t^2 e_4,  t  e_5,  -t  e_2 - e_3)} $&$
\lambda_4$&

$\lambda_6^{-\frac{1}{1+t}}$&$  \xrightarrow{(  
t e_1,  t^2 e_2,  t^3 e_3, t e_5,  -t^2 (t+1) e_4 - e_5)} $&$
 \lambda_5$ \\ \hline

$\mu_2  $&$\xrightarrow{(e_1, e_2, e_3, te_4, te_5)}$&$ \mu_1$ &

$\mu_{11}  $&$ \xrightarrow{( 
  t^2 e_1 - t^2 e_4,  
 t^4 e_2 + t^4 e_3, t^6 e_3, -t^3 e_2 - t^3 e_4,  t^5 e_5)}$&$
 \mu_2$  \\
\hline

$\mu_4$&$  \xrightarrow{(t^{-1}e_1, t^{-2}e_2, t^{-3}e_3, t^{-1}e_4, t^{-2}e_5)} $&$\mu_3$
&

$\mu_6$&$  \xrightarrow{(te_1, t^2e_2, t^3e_3, t^2e_4, t^3e_5)}$&$ \mu_5$
\\ \hline
 
$\mu_8^{\alpha} $&$ \xrightarrow{(e_1, e_2, e_3, te_4, te_5)} $&$\mu_7^{\alpha}$
&

$\mu_{11} $&$ \xrightarrow{(t^{-1}e_1, t^{-2}e_2, t^{-3}e_3, t^{-1}e_4, t^{-2}e_5)} $&$\mu_{10}$
\\ \hline

$\mu_{13}$&$  \xrightarrow{(t^{-1}e_1, t^{-2}e_2, t^{-3}e_3, t^{-1}e_4, t^{-2}e_5)}$&$ \mu_{12}$
&

$\mu_{15} $&$ \xrightarrow{(t^{-1}e_1, t^{-2}e_2, t^{-3}e_3, t^{-1}e_4, t^{-2}e_5)}$&$ \mu_{14}$\\
\hline
 
\end{longtable}

\begin{longtable}{ll}

\hline

$\mu_{22}^{t^{-1}} \to \mu_{4}$ &\\ 

$E^t_1= t^2 (t^4-1) e_1 + t^3 (t^2-1)^4 (t^2+1) e_5$ & 
$E^t_2= t^4 (t^4-1)^2 e_2 + t^4(t-1)^6   (t+1)^5 (t^2+1)^2 e_3 $\\ 

$E^t_3= t^6 (t^4-1)^3 e_3$& 
$E^t_4= t^4 (t-1)^4  (t+1)^3 (t^2+1) e_2 + ( t^4 + t^6)  e_4 $ \\
\multicolumn{2}{l}{$E^t_5= t^5(t-1)^2  (t+1) (t^2+1)^2 e_2 + 
 t^5 (t^2-1)^4 (t^2+1)^2 ( t^2-t - 1 ) e_3 + 
 t^5 (t^2 - 1) (t^2+1)^2 e_5$}\\

\hline

  $\mu_{15} \to \mu_6$ &\\ 

$E^t_1= \frac{t^2}{t+1} e_1$& 
$E^t_2= \frac{t^4}{(t+1)^2} e_2  $\\

$E^t_3= \frac{t^6}{(t+1)^3} e_3$& 
$E^t_4= \frac{t^3}{t+1} e_4 + \frac{t^4}{(t+1)^2} e_5$ \\
$E^t_5= \frac{t^5}{(t+1)^2} e_5$\\
 
 \hline

$\mu_{22}^{\alpha} \to \mu_8^{\alpha}$ &\\ 

$E^t_1= t (\alpha + \alpha^3) e_1 + (\alpha + \alpha^3) e_5$
& 
$E^t_2= t^2 (\alpha + \alpha^3)^2 e_2 + 
 t (1 - \alpha) (\alpha + \alpha^3)^2 e_3 $\\

$E^t_3= t^3 (\alpha + \alpha^3)^3 e_3$& 
$E^t_4= t (1 - \alpha) (\alpha + \alpha^3) e_2 + 
 t^2 (\alpha + \alpha^3) e_4 $  \\
$E^t_5= t^2 (1 - 2 \alpha) (\alpha + \alpha^3)^2 e_3 + 
 t^3 (\alpha + \alpha^3)^2 e_5$\\

\hline

$\mu_{11} \to \mu_{9}$&\\ 

$E_1^t= \frac{t^4+1}{t - t^2}e_1 - \frac{(t^4+1) (1 + t^5)}{2 (1 - t)^3 t^4}   e_2 + \frac{t^4+1}{(1 - t)^2} e_4$ &

$E_2^t= \frac{t^4+1}{(1 - t)^2 t^2} e_2 - \frac{(t^4+1)^2}{(1 - t)^4 t^5}
   e_3$ \\

$E_3^t=  \frac{(t^4+1)^2}{(1 - t)^3 t^3} e_3$  &
$E_4^t=  \frac{t^4+1}{(1 - t)^2 t^2}   e_2+\frac{ t^4+1}{(t-1)^2 t^2}e_4 $\\ 
   
$E_5^t=  \frac{t^4+1}{(1 - t)^2 t}e_2 +\frac{(t^4+1)^2}{(t-1)^4 t^4} e_5 $\\

\hline

  $\mu_{18} \to \mu_{13}$ &\\ 
$E^t_1= \sqrt{4 t-1} e_1 - 2  e_2 +  e_4$& 
$E^t_2= 4 t e_2 - 4  \sqrt{4 t-1} e_3 $\\
$E^t_3= 4 t  \sqrt{4 t-1} e_3$& 
$E^t_4= -2  e_2 + 2 t e_4 $ \\ 
$E^t_5= 2 t e_2 - 2 \sqrt{ 4 t-1} e_3 + 2 t \sqrt{4 t-1} e_5$\\

 \hline

  $\mu_{22}^{t} \to \mu_{19}$ &\\ 
$E^t_1= (t-1) (t + t^3) e_1 + t(t-1)^4  (1 + t^2) e_5$ & 
$E^t_2= (t-1)^2 (t + t^3)^2 e_2 - (t-1)^6 (t + t^3)^2 e_3 $\\

$E^t_3= (t-1)^3 (t + t^3)^3 e_3$& 
$E^t_4= (t-1)^4 (t + t^3) e_2 - t (1 + t^2) e_4  $ \\ 

\multicolumn{2}{l}{$E^t_5= (t-1)^4 (  2 t-1) (t + t^3)^2 e_3 - (t-1) (t + t^3)^2 e_5$}\\
\hline

$\mu_{22}^{i+t} \to\mu_{21}^{i}$ & \\

\multicolumn{2}{l}{
$\begin{array}{lcl}

E^t_1&=&
(i-1 + t)^2  (2it + t^2) ( 2 i (i + t) e_2+e_4)
\\

E^t_2&=&
(i-1+ t)^4  (2 it  + t^2)^2 (- (i + t) e_2 + 2 (i + t) (2 i + 2 t - i t^2)e_3+(1 + i + t) e_5  )\\

E^t_3&=&
-(i-1  + t)^6  (i + t) (2 it + t^2)^3 (i+1  + t) e_5\\

E^t_4&=
& {\tiny \frac{i+1}{2(i + t)}  
\left(
\begin{array}{l}(-8 t + (8 + 28 i) t^2 + (40 - 24 i) t^3 - (30 + 30 i) t^4 - (12 - 20 i) t^5 + (7 + 2 i) t^6 - i t^7) e_1+\\
((-8 - 8 i) t - (20 - 36 i) t^2 + (60 + 12 i) t^3 - (6 + 46 i) t^4 - (16 - 8 i) t^5 + (2 + 2 i) t^6) e_2+\\
((12 + 4 i) t - (6 + 42 i) t^2 - (52 - 30 i) t^3 + (35 + 33 i) t^4 + (12 - 21 i) t^5 - (7 + 2 i) t^6 + i t^7) e_4+\\
((-8 + 8 i) t + (52 + 20 i) t^2 - (4 + 108 i) t^3 - (98 -       46 i) t^4 + (44 + 40 i) t^5 + (6 - 16 i) t^6 - 2 t^7) e_5
\end{array}
\right)

      }\\

E^t_5&=&
(1 + i) (1 + i - it)   (i + t)^2 (2 + 2 i + (1 - 3 i) t -t^2)^2 t^2 e_2 - \\
&&(2 + 2 i)  (i + t)^2 (2 + 2 i + (1 - 3 i) t -     t^2)^2 (1 + i - (2 + 2 i) t - (1 - 3 i) t^2 + t^3)t^2  e_3+\\
&&\multicolumn{1}{r}{(1 - i) (i-1 + t)^2   (i+1 + t) (2 + 2 i + (1 - 3 i) t -t^2)^2t^2 e_5}  
\end{array}

      $}\\
      \hline
      
      $\mu_{22}^{-i+it} \to\mu_{21}^{-i}$ & \\

\multicolumn{2}{l}{
$\begin{array}{lcl}
E^t_1 & = & 
(t-2) (i-1 + t)^2 t e_4 + 
 2 (t-2) (t-1 ) (i-1 + t)^2 t e_5 \\

E^t_2 & = & 
 (2 - t)^2
 \left( 
\begin{array}{l}
i (1 - t) (i-1  + t)^4 t^2 e_2 - 
  2 (1 - t)  (i-1  + t)^4  (1 + i - (2 + i) t + 
    t^2)  t^2e_3+\\ 
    \multicolumn{1}{r}{i  (t-1 - i) (i-1 + t)^4 t^2 e_5}
    \end{array}
    \right)\\
 
E^t_3&=&
(t-2)^3 (t-1) (t-1 - i) (i-1 + t)^6 t^3 e_3\\

E^t_4&=&
\frac{(1 + i)(t-2) (i-1  + t)^2 t}{2(t-1)}
\left(
\begin{array}{l}
 (2 - 4 t +3 t^2 - t^3) e_1-
  (3 - i - (5 - i) t + (3 - i) t^2 - t^3) e_4+\\
\multicolumn{1}{r}{ (2 - (4 + i) t + (2 +i ) t^2) e_5}
\end{array}
   \right)\\

E^t_5 & =& 

(1 - i) (2 - t)^2 (1 - t)^2 (i-1  + t)^3 t^2 e_2 + 
(2 + 2 i) (2 - t)^2 (1 - t)^4 (i-1  + t)^3 t^2 e_3+\\
&& \multicolumn{1}{r}{(i-1 ) (2 - t)^2 (i-1  + t)^3   (2 - 2 t + t^2)t^2 e_5} \end{array}$}\\

\hline

$\mu_{17}^{t} \to{\mathfrak V}_{2+3}$ &\\ 

\multicolumn{2}{l}{
$\begin{array}{lcl}
E^t_1& =&\frac{t}{2 \lambda\sqrt{(1 + 4 \lambda)^3}}\Big((1 + 6 \lambda+ 8 \lambda^2) e_1 - \sqrt{(1 + 4 \lambda)^3} e_4  + \lambda e_5\Big) \\

E^t_2& =&\frac{t}{4 \lambda^2 \sqrt{(1 + 4 \lambda)^3}}
\Big( 
(1+ 4 \lambda) e_1 - \sqrt{(1 + 4 \lambda)^3} e_4 + 2 \lambda e_5
\Big) \\

E^t_3& =&  -\frac{t^2}{4 \lambda^2  \sqrt{(1 + 4 \lambda)^3}}
\Big( 
 - 
 2 (1 + 3 \lambda) \sqrt{1 + 4 \lambda} e_2 + (1 + 7 \lambda) e_3+2 (1 + 5 \lambda + 4 \lambda^2) e_5\Big) \\

E^t_4& =&  -\frac{t^2}{4 \lambda^2 \sqrt{(1 + 4 \lambda)^3}}
\Big( 
 - 
 2 (1 + 3 \lambda) \sqrt{1 + 4 \lambda} e_2 + (1 + 9 \lambda + 8 \lambda^2) e_3+2 (1 + 5 \lambda + 4 \lambda^2) e_5\Big) \\

E^t_5& =&  
-\frac{t^2}{4 \lambda^2  \sqrt{(1 + 4 \lambda)^3}}
\Big( 
- 
 2 (1 + 2 \lambda) \sqrt{1 + 4 \lambda} e_2 + (1 + 8 \lambda) e_3+(2 + 8 \lambda) e_5 \Big)

\end{array}$}\\

\hline

\end{longtable}

After  careful  checking  dimensions of orbit closures for the rest of algebras, we have 

\begin{center}  
$\dim \mathcal{O}(\mu_{1,4}^5)=\dim \mathcal{O}(\lambda_{6}^\alpha)=\dim \mathcal{O}(\mu_{11}) = 
\dim \mathcal{O}(\mu_{15})=\dim \mathcal{O}(\mu_{17})=$

$\dim \mathcal{O}(\mu_{18})=\dim \mathcal{O}(\mu_{20})=\dim \mathcal{O}(\mu_{22}^\alpha)=
\mathcal{O}({\mathfrak V}_{4+1})=20,$ 
$\dim \mathcal{O}({\mathfrak V}_{3+2})=24.$ 

 \end{center}
 Hence, 
$\mu_{1,4}^5, \lambda_{6}^\alpha, \mu_{11}, \mu_{15},  \mu_{17},  \mu_{18}, \mu_{20},  \mu_{22}^\alpha, {\mathfrak V}_{4+1}$ and ${\mathfrak V}_{3+2}$ give $10$ irreducible components.

\end{Proof}

\begin{Remark}
Our Theorem C talks that the variety of complex $5$-dimensional nilpotent associative algebras has $11$ irreducible components. We defined  dimensions of all irreducible components, generic and rigid algebras in this variety.
It is improving some early results in this direction.
Namely, the result from \cite{mak93} talks that in this variety there are $13$ irreducible components.
Also, the result from \cite{mak93} does not talk about dimensions of irreducible components, generic algebras, and rigid algebras of this variety.
On the other hand, there are some inaccuracies in the definition of irreducible components in \cite{mak93}. 
For example, 
the algebra from the case [(a), page 36], 
has nilindex 2 and the basis $\{ x, x^2, y, xy, x^2y\}.$

\end{Remark}


\begin{thebibliography}{99}
 

 \bibitem{ale3}
Alvarez M.A., 
Degenerations of $8$-dimensional $2$-step nilpotent Lie algebras,
Algebras and Representation Theory, 2020, DOI: 10.1007/s10468-020-09987-5
 
 
 \bibitem{aut}
Autenried C., Furutani K., Markina I., Vasil'ev A., 
    Pseudo-metric 2-step nilpotent Lie algebras,
    Advances in Geometry, 18 (2018), 2, 237--263.

\bibitem{cal}
Calder\'on A.,  Fern\'andez Ouaridi A., Kaygorodov I.,  
On the classification of bilinear maps with radical of a fixed codimension, Linear and Multilinear Algebra, 2020, DOI:10.1080/03081087.2020.1849001

\bibitem{BC99} 
Burde D., Steinhoff C.,
    Classification of orbit closures of $4$--dimensional complex Lie algebras,
    Journal of Algebra, 214 (1999), 2, 729--739.
 

 \bibitem{chouhy}
Chouhy S.,
    On geometric degenerations and Gerstenhaber formal deformations,
    Bulletin of the London Mathematical Society, 51 (2019),  5, 787--797.
  
  
 \bibitem{cibils}  Cibils C., 
    $2$-nilpotent and rigid finite-dimensional algebras,
    Journal of the London Mathematical Society (2), 36 (1987), 2, 211--218.


  \bibitem{fF68}
Flanigan F.\ J., 
    Algebraic geography: {V}arieties of structure constants, 
    Pacific Journal of Mathematics, 27 (1968), 71--79.
      
     \bibitem{fb}
     Farinati M.,  Jancsa A., 
     Lie bialgebra structures on $2$-step nilpotent graph algebras,
        Journal of  Algebra, 505 (2018), 70--91.
        
        
  \bibitem{gabriel}
Gabriel P.,
Finite representation type is open,
Proceedings of the International Conference on Representations of Algebras (Carleton Univ., Ottawa, Ont., 1974), pp. 132--155.
 
 \bibitem{ger63}
Gerstenhaber M.,
    On the deformation of rings and algebras,
    Annals of Mathematics (2), 79 (1964), 59--103.
	
 \bibitem{gorb93} 
Gorbatsevich V., 
    Anticommutative finite-dimensional algebras of the first three levels of complexity, 
    St. Petersburg Mathematical Journal, 5 (1994), 505--521.
 
 
\bibitem{GRH}
Grunewald F.,  O'Halloran J.,
    Varieties of nilpotent Lie algebras of dimension less than six,
    Journal of Algebra, 112 (1988), 2, 315--325.
    
\bibitem{hs08}
Horn R., Sergeichuk V.,
    Canonical matrices of bilinear and sesquilinear forms,
    Linear Algebra and its Applications, 428 (2008), 193--223.
    
\bibitem{ikm20}
    Ismailov N.,    Kaygorodov I.,     Mashurov F., 
    The algebraic and geometric classification of nilpotent assosymmetric algebras, 
    Algebras and Representation Theory, 24 (2021),  1, 135--148.  
 
\bibitem{karim20}
Karimjanov I.,
The classification of 5-dimensional complex nilpotent associative algebras,
Communications in Algebra, 49 (2021),  3, 915--931.

\bibitem{kaladra20} 
Karimjanov I.,  Ladra M., 
Some classes of nilpotent associative algebras,
Mediterranean Journal of Mathematics, 17 (2020), 70.   


\bibitem{klp20}
Kaygorodov I., Lopes S., Popov Yu.,
Degenerations of nilpotent associative commutative algebras,
Communications in Algebra,    48  (2020), 4, 1632--1639.



\bibitem{kkl21}
Kaygorodov I., Khrypchenko M., Lopes S.,
    The geometric classification of nilpotent  algebras,
     arXiv:2102.10392
    
    \bibitem{kkp20}
Kaygorodov I., Khrypchenko M., Popov Yu., 
The algebraic and geometric classification of nilpotent terminal algebras, Journal of Pure and Applied Algebra,  225 (2021), 6, 106625.


\bibitem{kppvcor}
Kaygorodov I., Popov Yu., Pozhidaev A., Volkov Yu.,
Corrigendum to  "Degenerations of  Zinbiel and nilpotent Leibniz  algebras", 
Linear and Multilinear Algebra, 2020, DOI:  10.1080/03081087.2020.1749543



 \bibitem{kv17}
Kaygorodov I., Volkov Yu., 
    Complete classification of algebras of level two,  
    Moscow Mathematical Journal, 19 (2019), 3,  485--521.

\bibitem{lo14}  
Lauret J.,  Oscari D., 
    On non-singular $2$-step nilpotent Lie algebras, 
     Mathematical Research Letters, 21 (2014), 3, 553--583.
 
  
      
\bibitem{mak93}
Makhlouf A.,  
    The irreducible components of the nilpotent associative algebras,  
    Revista Matemática Complutense, 6 (1993), 1, 27--40.




\bibitem{maz79}
Mazzola G.,
    The algebraic and geometric classification of associative algebras of dimension five, 
    Manuscripta Mathematica, 27 (1979), 1, 81--101. 
     
\bibitem{shaf}
    Shafarevich I., 
    Deformations of commutative algebras of class $2,$ Leningrad Mathematical Journal, 2 (1991), 6, 1335--1351.
 
 
 

\end{thebibliography}
\end{document}